\documentclass{amsart}

\usepackage{amsmath, amsfonts, amssymb,amsthm}
\usepackage{amssymb,latexsym}

\newtheorem{theorem}{Theorem}
\newtheorem{lemma}[theorem]{Lemma}
\newtheorem{cor}[theorem]{Corollary}

\theoremstyle{definition}

\newtheorem*{defi}{Definition}
\newtheorem*{remark}{Remark}

\let\d\delta

\def\CC{\mathbb C}

\def\PP{\mathbb P}

\def\Ocal{\mathcal O}
\let\frak\mathfrak

\def\KK{{ K}}

\begin{document}
\title{non-archimedean analytic curves in the complements of
hypersurface divisors}
\author{Ta Thi Hoai An$^*$}
\address{Institute of Mathematics\\
18 Hoang Quoc Viet, Cau Giay\\
Hanoi, VietNam}
\email{tthan@math.ac.vn}
\author{Julie Tzu-Yueh Wang}
\address{Institute of Mathematics\\Academia
Sinica\\Nankang\\ Taipei 11529\\Taiwan, R.O.C.}
\email{jwang@math.sinica.edu.tw}
\author{Pit-Mann Wong}
\address{Department of Mathematics\\University of Notre Dame\\Notre
Dame\\
IN 46556\\U.S.A.}
\email{wong.2@nd.edu}
\thanks{*Junior Associate Member of ICTP, Trieste, Italy}
\begin{abstract}
 We study the degeneration dimension of   non-archimedean analytic maps
into the complement of hypersurface divisors of   smooth projective
varieties.   We also show that there exist no non-archimedean analytic
maps into
$\PP^n\setminus\cup_{i=1}^n D_i$ where $D_i$, $1\le i\le n$, are
hypersurfaces of degree at least 2 in general position and intersecting
transversally.  Moreover, we prove that  there exist no
non-archimedean analytic maps into  $\PP^2\setminus\cup_{i=1}^2 D_i$
when $D_1$, $D_2$ are generic plane curves with $\deg D_1+\deg D_2\ge 4$.
\end{abstract}
\baselineskip=16truept \maketitle \pagestyle{myheadings}

\markboth{T. An, J. Wang and P.-M. Wong}{Non-archimedean Hyperbolicity}

\section{Introduction}
A complex manifold $X$ is said to be hyperbolic (in the sense of Brody)
if every
analytic map from the complex plane $\CC$ to $X$ is constant. One of
the
main
related questions is when the complements of hypersurfaces in $\PP^n$
are
hyperbolic.
It was conjectured by Kobayashi \cite{Ko} and Zaidenberg \cite{Z} that
the complements of ``generic" hypersurfaces in $\PP^n$ with degree at
least $2n+1$
are hyperbolic.
There have been many results related to this conjecture.  We will
only mention a few that are related to our results in the non-archimedean
case.  Green \cite{G} verified the conjecture
 in the case of
$2n+1$ hyperplanes in
general position. More generally, Babets \cite{Ba}, Eremenko-Sodin
\cite{ES}, and Ru \cite{Ru2} independently  showed that
$\PP^n\setminus\{\,2n+1\text{ hypersurfaces in general position}\}$ is
hyperbolic.
When $n=2$, the conjecture is correct for the case of four
generic curves (cf. \cite{DSW1}).
For the case of three  generic  curves $C_1,\,C_2,\,C_3$,
Dethloff,
Schmacher, and Wong (\cite{DSW1}, \cite{DSW2}) showed for that
$\PP^2\setminus
\cup_{i=1}^3C_i $ is hyperbolic if $\deg C_i\ge 2$ for $i=1,2,3$. When
one of the
$C_i$ is a line they show that any holomorphic map
$f:\CC\to\PP^2\setminus
\cup_{i=1}^3C_i $ is algebraically degenerate if $d_1=1$, $d_2\ge 3$ and
$d_3\ge 4$, up to enumeration.  More generally, one can estimate the
dimension of the image of an analytic map from $\CC$ into a complex
manifold.  For example, it is well-known that
an analytic map $f:\CC\to \PP^n\setminus\{\,n+2\text{ hypersurfaces in
general position}\}$ is algebraically degenerate.  Recently, more
studies in this direction for varieties off ample divisors were
studied by Noguchi and Winkelmann in \cite{NW}.

Similar questions can be asked when the ground field is
non-archimedean.
Let $\KK$ be an algebraically closed fields of arbitrary
characteristic,
complete
with respect to a non-archimedean absolute value $| \, \, |$.
A variety $X$ over $\KK$ is said to be $\KK$-hyperbolic if every
analytic
map from $\KK$ to $X$ is constant. In contrast to the situation over
the
complex
numbers, it is much easier to study hyperbolic problems over
non-archimedean ground
fields. For example,
as an easy consequence of the second main theorem of Ru \cite{Ru},
$\PP^n\setminus\{\, n+1\text{ hypersurfaces in general position}\}$ is
$\KK$-hyperbolic. Similarly, the second main theorem of An
\cite{A} shows that   $X\setminus\{\, n+1\text{
hypersurface divisors in general position}\}$ is
$\KK$-hyperbolic where $X$ is a projective variety over $\KK$ in $\PP^N$
and a hypersurface divisor is the intersection of a hypersurface  in
$\PP^N$ with $X$. It is
also easy to see that an analytic map
$f:\KK\to \PP^n\setminus D_1\cup D_2$ is algebraically degenerate if $D_1$
and $D_2$ are
distinct hypersurfaces in $\PP^n$.

In this note, we will first study the algebraic degeneracy of
        non-archimedean analytic maps omitting hypersurface divisors in
smooth  projective
        varieties.  The results are stated in section 2. Second, we
will
        study $K$-hyperbolicity for complements of hypersurfaces in
    projective space.  Similar to the conjecture of Kobayashi and
    Zaidenberg, we conjecture the following:

\medskip
\noindent{\bf Conjecture.}
Let $D_1$,...,$D_q$, $q\le n$, be $q$ distinct generic hypersurfaces in
$\PP^n(\KK)$. If
$\sum_{i=1}^q\deg D_i\ge 2n$, then
$\PP^n\setminus \cup_{i=1}^q D_i $ is $\KK$-hyperbolic.

\medskip
In Section 3, we will prove this conjecture for the case of $n$ generic
hypersurfaces and $\deg D_i\ge 2$ for each $1\le i\le n$.   We will also
prove this conjecture for $\PP^2$ omitting 2 generic curves.

 {\it Acknowledgments.} A part of this article was written while the first and the third name
authors were visiting the Institute of Mathematics, Academia Sinica,
Taiwan.  They would like to thank the Institute for warm hospitality.

\section{Results in General cases}

\medskip
Let $X$ be a $n$-dimensional projective
variety. A collection of $q $ effective divisors
$D_1,\dots,D_q$ of $X$ is said to be
\textit{in general position } if
\begin{enumerate}
\item[(i)] $q\le n$, and the codimension of each component of
$\cap_{j=1}^q D_{q}$ in
$X$ is $q$;
\item[(ii)] $q\ge n+1$, and for any subset
$\{i_0,\dots,i_n\}$ of $\{1,\dots,q\}$ of cardinality $n+1,$
$\cap_{j=0}^n D_{i_j} =\emptyset.$
\end{enumerate}

\begin{theorem}\label{degenerate} Let $X$ be an $n$-dimensional
nonsingular projective
subvariety of
$ \PP^N(\KK)$ and $P_1, ..., P_q$, $ q\ge 2$,
be non-constant homogeneous polynomials in $N+1$ variables.
Let $D_i=X\cap\{P_i=0\}$, $1\le i\le q$, be the divisors of $X$
in general position.  Let $f$ be an analytic map from $\KK$ to
$X\setminus \cup_{i=1}^qD_i$.
Then the image of $f$ is contained in a subvariety
of $X$ of codimension $\min\{n+1, q\}-1$ in $X$.
In particularly, $f$ is algebraically degenerate if $q\ge2$, and
$X\setminus \cup_{i=1}^qD_i$ is $\KK$-hyperbolic if $q\ge n+1$.
\end{theorem}

The following example shows that the theorem is sharp.

\medskip
\noindent
{\bf Example.}
Let $X=\PP^n$ and  $D_1$,...,$D_q$, $q\le n$, be the
coordinate hyperplanes $\{X_{n-q+1}=0\}$,...,$\{X_{n}=0\}$.
Let $f_0$,...,$f_{n-q}$ be algebraically independent $K$-analytic
functions. Then
$f=(f_0,f_1,...,f_{n-q},1,...,1)$ is a non-constant analytic map into
$\PP^n\setminus
\cup_{i=1}^qD_i$, and the codimension of  the Zariski closure of its image
is
$q-1$.

We will need the following well-known result of non-archimedean Picard's
theorem.

\begin{lemma}\label{curve} Let
$C$ be a irreducible projective curve. Then $C \setminus \{\text{two
distinct
points}\}$ is $\KK$-hyperbolic.
\end{lemma}

\begin{proof}[Proof of Theorem \ref{degenerate}]
Denote by $d_i:=\deg P_i$ and let $l$ be the least common multiple of
$d_i$, $1\le i\le q$. Without loss of generality, we may assume $\deg
P_i=d$ for each $1\le i\le q$ after replacing
$P_i$ by $P_i^{\frac{l}{d_i}}$.

Let $(f_0,...,f_N)$ be a reduced
representation of
$f$, i.e.,
$f_0,...,f_N$,
are $\KK$-analytic functions with no common zero and $f=[f_0:...:f_N]$.
Then
$f :\KK \to X\setminus \cup_{i=1}^q D_i$ implies that
$P_i(f_0,...,f_N)$ $(1\le i\le q)$ is an analytic function with no
zero.
By the non-archimedean Picard's theorem, $P_i(f_0,...,f_N)$ is a
non-zero
constant for each
$1\le i\le q$.
Then there exist non-zero constants
$c_2,...,c_{q}$ such that
$P_{i} (f(z))-c_{i}P_1 (f(z))=0$
for all $z\in \KK$ and
$2\le i\le q$.
Let $W_i$ be the algebraic subset in $\PP^N$ defined by
$P_{i} -c_{i}P_1 =0$.
Then the image of
$f$ is contained in
$W=X\cap_{i=2}^{q} W_i$.
We will show that the codimension of each component of $W$ is
$q-1$ if $q\le n$ and $W$ is a finite set of points  if $q\ge n+1$.
We first consider when $q\le n$. In this case we have
$\cap_{i=1}^{q}D_i$ is not empty.
Let $x\in \cap_{i=1}^{q}D_i$.  Since $X$ is irreducible, it suffices to
compute the codimension of $W$ in an open subset of $X$.  In
particularly, we take an open neighborhood $U_x$ of $x$.  Rearranging the
coordinates if necessary, we may assume that $X_0(x)\ne 0$, and denote by
$\bar P_i=P_{i} /X_0^{d}$ for each $1\le i\le q$. Then the defining
equation of
$W_i$ in
$U_x$ can be replaced by
\begin{align}\label{gi}
\d_i:=\bar P_{i}
-c_{i}\bar P_1,\qquad  2\le i\le q.
\end{align}
We will show that $[\bar P_1,\d_2,...,\d_q]$ is a regular sequence of
the local ring $\Ocal_{X,x}$, and therefore so is $[\d_2,...,\d_q]$
which implies the codimension of each component of $W$ in $X$ is $q-1$.
 For $2\le r\le q$, it is clear that
$(\bar P_1,\d_2,...,\d_r)=(\bar P_1,..., \bar P_r)$ as ideals in
$\Ocal_{X,x}$.   Suppose that $\delta_r$ is a zero divisor
of $\Ocal_{X,x}/(P_1, \delta_2, ..., \delta_{r-1}).$
Then there
exist $G\in \Ocal_{X,x}$ and $G\not\in(\bar P_1, \delta_2,
..., \delta_{r-1})$ such that $G(\bar P_r-c_r\bar P_1)\in (\bar P_1,
\delta_2, ..., \delta_{r-1})=(\bar P_1,..., \bar P_r)$. This then implies
that  $G\bar P_r\in (\bar P_1, ...,\bar P_{r-1})$.
Since $G\not\in(\bar P_1, \delta_2,
..., \delta_{r-1})=(\bar P_1, ...,\bar P_{r-1})$, this means that
$[\bar P_1, ..., \bar P_{q}]$ is not a regular sequence of $\Ocal_{X,x}$.
However, since $D_1,...,D_q$ are in general
position,  $[\bar P_1, ..., \bar P_{q}]$ is a regular sequence of
$\Ocal_{X,x}$.  This gives a contradiction and shows that
$[\bar P_1,\d_2,...,\d_q]$ is a regular sequence of
the local ring $\Ocal_{X,x}$

When $q\ge n+1$, it follows from the previous arguments that the codimension of
$X\cap W_2\cap
\cdots\cap W_{n}$ is $n-1$.  Hence, the image of $f$ is contained
in an irreducible curve $C$.
 Since $q\ge n+1$ and $D_i$, $1\le i\le q$, are in general position,
it is clear that
$  C\cap (\cup_{i=1}^q D_i )$ contains at least two distinct points.
Therefore, the image of $f$ is contained in the curve $C$ omitting
 two points, and hence $f$ is constant by
  Lemma~\ref{curve}.
\end{proof}


\section{Results in projective spaces}

\begin{defi}
Nonsingular hypersurfaces $D_1, ..., D_n$ in $\PP^n(\KK)$  {\it
intersect
transversally} if for every point $x\in\cap_{i=1}^n D_i$,
$\cap_{i=1}^n\Theta_{D_i,x}=\{x\}, $
where $\Theta_{D_i,x}$ is the tangent space to $D_i$  at $x$.
\end{defi}

\begin{theorem}\label{transversal} Let
$D_1, ..., D_n$ be nonsingular hypersurfaces in $\PP^n(\KK)$ intersecting
transversally.
Then $\PP^n\setminus \cup_{i=1}^n D_i$ is $\KK$-hyperbolic if  \ $\deg
D_i\ge 2$ for each $1\le i\le n$.
\end{theorem}

\begin{remark}
The assumption on the degree of the hypersurfaces is sharp. For
example, in the
case of $\PP^2$, we may choose $D_1=\{X_0=0\}$ and
$D_2=\{X_0^2+X_1^2-X_2^2=0\}$
and let $f(z)=(1,z,z)$. It is easy to check that $D_1$ and $D_2$
intersect transversally, and it is clear that
$f$ is a non-constant analytic map into
$\PP^2\setminus\{D_1\cup D_2\}.$  For general $\PP^n$, we may choose
$D_1=\{X_0=0\}$, and $D_i=\{ X_0^2+a_{i1}X_1^2+\cdots+a_{in}X_n^2=0\}$
with
$a_{i1}+\cdots+a_{in}=0$ for $2\le i\le n$. Furthermore,  we may assume
that every $n-1$ by $n-1$ submatrix of the matrix
$(a_{ij})_{i,j}$, ${2\le i\le n}, {1\le j\le n}$, has rank $n-1$. Then
these hypersurfaces intersect transversally.
Clearly, the
analytic map $f(z)=(1,z,z,...,z)$ does not intersect any of the
hypersurfaces $D_i$, $1\le i\le n$.
\end{remark}

\begin{proof}[Proof of Theorem \ref{transversal}]
Assume that there exits an analytic map $f:\KK\to \PP^n\setminus
\cup_{i=1}^n D_i$.
By Theorem~\ref{degenerate}, we see that the image of $f$ is contained
in an
irreducible curve $C$. Then $f$ is an analytic map from $\KK$ into
$C\setminus
\cup_{i=1}^n D_i$. If
$C\cap\{\cup_{i=1}^n D_i\}$ consists at least two points, then it
follows from
Lemma~\ref{curve} that $f$ is a constant. It then remains to consider
when
$C\cap\{\cup_{i=1}^n D_i\}$ consists exactly only one point $x$. Since
$\dim C+\dim D_i-n\ge 0$, we have
$C\cap
D_i\not=\emptyset$ for each $i$. This can only happen when $x\in
\cap_{i=1}^n D_i$
and $C\cap D_i=\{x\}$ for each $i$. By Bezout's theorem, we have
$$
(C,D_i)_x=\deg C.\deg D_i\ge \deg D_i\ge 2
$$
for each $i$.
Therefore, $\Theta_{C,x}\cap\Theta_{D_i,x}\supsetneq\{x\}$ for each
$i$.
Since
$\cap_{i=1}^n\Theta_{D_i,x}=\{x\}$, this shows that $C$ must have at
least
2
different tangent lines at the point $x$. Let
$\pi : \tilde{C} \to C$ be the normalization of $C$. Then
$\# \pi^{-1}(x) \ge 2$. If $f : \KK \to C \setminus \{x\}$
then its lifting $\tilde{f} : \KK \to \tilde{C}$ misses
$\pi^{-1}(x)$ containing at least 2 points thus $\tilde{f}$ is a
constant and so $f = \pi \circ \tilde{f}$ is a constant.
\end{proof}

For the rest of this section we consider the particular case when $n=2$.

\begin{defi} Let $D$ be a curve of degree $d\ge 3$ in
$\PP^2$.
A nonsingular point $x$ of $D$ is said to be a
{\it maximal inflexion point}
if there exits a line intersecting $D$ at $x$ with multiplicity $d$.
\end{defi}

\begin{remark} The curve $X^d-YZ^{d-1}=0$ has  a maximal inflexion point
$P=(0,0,1)$ if $d\ge 3$.  Every smooth cubic has 9 maximal inflexion
points counting multiplicities which are indeed the inflexion points.
Since a maximal inflexion point is an inflexion point, the coefficients
of the defining equation of the curve need to satisfy an algebraic
equation (i.e. its Hessian form cf.
\cite{Sa}). Therefore, it is not difficult to see that a generic curve of
degree
$d\ge 4$ has no maximal inflexion points.
\end{remark}

 For simplicity of notation, we will use $D$ to denote
the defining
equation of the curve $D$.
We have the following.

\begin{theorem}\label{P2}
Let $D_1$ and $D_2$ be nonsingular projective curves in $\PP^2$.
Assume that $D_1$ and $D_2$ intersect transversally and $\deg D_1\le
\deg
D_2$.
Then $\PP^2\setminus\{D_1\cup D_2\}$ is $\KK$-hyperbolic if and only if
either
$\deg D_1, \deg D_2\ge 2$ or $\deg D_1=1$, $\deg D_2\ge 3$ and $D_1$
does not
intersect $D_2$ at any maximal inflexion point.
\end{theorem}

The non-archimedean analogue of the Kobayashi-Zaidenberg conjecture for
the case of $\PP^2$ omitting two generic curves follows directly.

\begin{cor}\label{generic}
Let $D_1$ and $D_2$ be distinct generic curves in $\PP^2$.
If $\deg D_1+\deg D_2\ge 4$ then
$\PP^2\setminus\{D_1\cup D_2\}$ is $\KK$-hyperbolic.
\end{cor}

\begin{proof}
Since two curves intersecting transversally is a generic condition, by the
previous theorem we only need to verify that for generic curves $D_1$ and
$D_2$ with
$\deg D_1=1$,
$\deg D_2\ge 3$,  $D_1$
does not
intersect $D_2$ at any maximal inflexion point.
 A generic curve $D_2$ of  degree at least 4 has  no  maximal
inflexion point.
When the curve $D_2$ has degree 3 then
it has at most 9 maximal inflexion points.   In this case, a generic
line $D_1$ does not intersect $D_2$
 at any maximal inflexion points.
The corollary then follows directly from Theorem \ref{P2}.
\end{proof}

To prove Theorem 4, we first study some cases that
$\PP^2\setminus\{D_1\cup D_2\}$ fails to be $\KK$-hyperbolic.

\begin{lemma}\label{smalldegree}
$\PP^2\setminus\{D_1\cup D_2\}$ is not $\KK$-hyperbolic if
\begin{enumerate}
\item[(i)] $\deg D_1=1$ and $\deg D_2\le 2$ or
\item[(ii)] $\deg D_1=\deg D_2=2$ and $D_1$ and $D_2$ intersect
tangentially.
\end{enumerate}
\end{lemma}

\begin{proof}
We will construct a non-constant analytic map $f$ from
$\KK$ into $\PP^2\setminus\{D_1\cup D_2\}$ which implies that
$\PP^2\setminus\{D_1\cup D_2\}$ is not $\KK$-hyperbolic.
We first consider when $\deg D_1=\deg D_2=1$.  Let $\frak{p}$ be the
intersection point of $D_1$ and $D_2$ and $L\ne D_i$, $i=1,2$, is a line
passing through
$\frak{p}$. Then we can construct a non-constant analytic  map $f$ from
$\KK$ into
$\PP^2$ such that its image is contained  $L\setminus \frak{p}$.  Then
this gives a non-constant analytic map into
$\PP^2\setminus\{D_1\cup D_2\}$.  The following is an explicit
construction. By linear change of coordinates, we may assume that
$D_1=\{X_0=0\}$ and
$D_2=\{X_1=0\}$. Let
$f(z)=(1,1,z)$, then $D_1(f)=D_2(f)=1$. In other words, $f$ is a
non-constant
analytic map from $\KK$ into $\PP^2\setminus\{D_1\cup D_2\}$.
Indeed, its image is in $L\setminus \frak{p}$ where $L=\{X_0-X_1=0\}$
and  $\frak{p}=(0,0,1).$

Next, we consider when $\deg
D_1=1$ and $\deg D_2=2$.  In this case, the intersection of
$D_1$ and $D_2$ contains at most two distinct points.
If the intersection contains only one point $\frak{p}$,
then  $D_1$ is tangent to $D_2$ at $\frak{p}$.
We can produce a non-constant analytic map such that its image is  a
conic tangent to $D_1$ at $\frak{p}$ and omitting $\frak{p}$.
Indeed, after change
of coordinates, we may assume that $D_1=\{X_0=0\}$ and $\frak p=(0,0,1)$.
Both conditions force the defining condition of $D_2$ to be of the form
$\{a_0X_0^2+a_1X_1^2+a_2X_0X_1+a_3X_0X_2=0\}$.
Since $D_2$ is non-singular, it has to be irreducible. Therefore, we
may
further
assume that $a_1\ne 0$ and $a_3\ne 0$. Without loss of generality, we
let $a_3=1$.
Now let $f(z)=(1,z,1-a_0-a_2z-a_1z^2)$. Then $D_1(f)=D_2(f)=1$, which
shows that
$f$ is a non-constant
analytic map from $\KK$ into $\PP^2\setminus\{D_1\cup D_2\}$.
If $D_1$ and $D_2$ contains two points, i.e. $D_1$ is not a tangent
line of $D_2$.  In this case, we may produce a non-constant analytic map
whose image is  the line tangent to the conic $D_2$ at one of the
intersection points of $D_1$ and $D_2$ and omitting the intersection
point.  We will skip the explicit  construction since it is similar to the
previous one.

For case (ii),  as  $D_1$ and $D_2$ intersect tangentially, there is
exactly one intersection point.  Denote it by  $\frak{p}$.
We  can produce a non-constant analytic map such that its image
is contained in the common tangent line of two conics $D_1$ and $D_2$
and omits $\frak{p}$.  The explicit construction is similar to the
previous one, and we will omit it since the geometric meaning is
clear.
\end{proof}

\begin{proof}[Proof of Theorem \ref{P2}]
If $\deg D_1, \deg D_2\ge 2$ then it follows from
Theorem~\ref{transversal} that $\PP^2\setminus\{D_1\cup D_2\}$ is
$\KK$-hyperbolic.
We now consider when
$\deg D_1=1$, $\deg D_2\ge 3$ and $D_1$ does not intersect $D_2$ at any
maximal
inflexion point of $D_2$. Let
$f:\KK\to \PP^2\setminus \{D_1\cup D_2\}$ be an analytic map.
If it is non-constant, then by Theorem
\ref{degenerate}, the image of
$f$ is contained in an irreducible plane curve $C$. In other words, we
have an
analytic map $f:\KK\to C\setminus \{D_1\cup D_2\}$. By Lemma
\ref{curve},
$C\cap\{\cup_{i=1}^n D_i\}$ can consist of only one point, otherwise
$f$ must be constant. Let
$C\cap\{\cup_{i=1}^n D_i\}$ consist of only one point $\frak{p}$.
If $\deg C\ge 2$ has degree at least 2, then by Bezout's theorem, we
have
$$
(C,D_i)_\frak{p}=C\cdot D_i=\deg C\cdot \deg D_i\ge \deg C\ge 2
$$
for $i=1,2$. Then the arguments in the proof of
Theorem~\ref{transversal} show that
$f$ must be a constant. Therefore, it remains to consider when $\deg
C=1$. In this
case, $(C,D_2)_\frak{p}=\deg D_2$ and $\frak{p}\in D_1\cap D_2$. This
implies that
$\frak{p}$
is a
maximal inflexion point of $D_2$ and $D_1$ intersects $D_2$ in $\frak{p}$,
which is
impossible by the hypothesis. In conclusion, $f$ must be constant.

For the converse part, we have shown in Lemma \ref{smalldegree} that
$\PP^2\setminus\{D_1\cup D_2\}$ is not $\KK$-hyperbolic if $\deg D_1=1$
and $\deg
D_2\le 2$. It then remains to show for the case when
$\deg D_1=1$, $\deg D_1\ge 3$ and $D_1$ does intersect $D_2$ at a
maximal
inflexion point of $D_2$.
Without loss of generality, we let $\frak{p}=(0,0,1)$ is a maximal
inflexion point of $D_2$
and $L=\{X_0=0\}$ is the tangent line of $D_2$ at $\frak{p}$. Then the
intersection
multiplicity of $(L,D_2)_\frak{p}$ equals $\deg D_2$, and
therefore
$L\cap D_2=\{\frak{p}\}$ by Bezout's theorem.
On the other hand, as $D_1$ intersects
$D_2$ transversally $D_1\ne L$.  Moreover, since $\frak{p}\in D_1$,
 $D_1$ is defined by $bX_0+X_1$. Let $ f=(0, 1, z):\KK\to \PP^2$.
Then
$D_1(f)=1$ and the image of $f$ is contained in $L$.
Since $L\cap D_2=\{\frak{p}\}$ and $\frak{p}=(0,0,1)\notin f(\KK)$,
$f(\KK)\cap D_2=\emptyset$.
This shows that in this case $\PP^2\setminus \{D_1\cup D_2\}$ is not
$\KK$-hyperbolic.
\end{proof}

\bibliographystyle{amsalpha}

\end{document}